\documentclass{amsart}

\usepackage{amssymb, array, verbatim, amscd}
\usepackage{amsmath, amsfonts,amsthm, enumerate}

\theoremstyle{plain}
\newtheorem{theorem}{Theorem}[section]
\newtheorem{corollary}[theorem]{Corollary}
\newtheorem{lemma}[theorem]{Lemma}
\newtheorem{proposition}[theorem]{Proposition}

\theoremstyle{definition}

\newtheorem{remark}[theorem]{Remark}

\newcommand{\N}{{\mathbb N}}

\newcommand{\Z}{{\mathbb Z}}

\newcommand{\Modr}{\mathrm{Mod}\text{-}}

\newcommand{\Hom}{\operatorname{Hom}}

\newcommand{\Add}{\operatorname{Add}}
\newcommand{\Prod}{\operatorname{Prod}}

\newcommand{\Htp}[1]{\mathbf{K}({#1})}

\begin{document}

\title{$\Sigma$-pure injectivity and Brown representability}

\author{Simion
Breaz}\thanks{Research supported by the
CNCS-UEFISCDI grant PN-II-RU-TE-2011-3-0065}
\address{ "Babe\c s-Bolyai" University, Faculty of Mathematics
and Computer Science, Str. Mihail Kog\u alniceanu 1, 400084
Cluj-Napoca, Romania} \email{bodo@math.ubbcluj.ro}
\date{\today}


\thanks{ }

\begin{abstract}
We prove that a right $R$-module $M$ is $\Sigma$-pure injective if
and only if $\Add(M)\subseteq \Prod(M)$. Consequently, if $R$ is a
unital ring, the homotopy category $\Htp{\Modr R}$ satisfies the
Brown Representability Theorem if and only if the dual category
has the same property. We also apply the main result to provide
new characterizations for right pure-semisimple rings or to give a
partial positive answer to a question of G. Bergman.
\end{abstract}


\maketitle
\date{}

\section{Introduction}

It is well known that purity plays a central role in the study of
module categories. In particular ($\Sigma$-)pure
injective/projective modules are used in order to describe various
properties for some subcategories of a module category. For
instance the pure-semisimplicity (i.e. all modules are
pure-injective) of $\Modr R$ can be characterized in many ways.
However, there are many open problems related to these notions.

We approach in this note some of these problems.  It was proved by
L. Angeleri, H. Krause and M. Saorin that $\Prod(M)\subseteq
\Add(M)$ if and only if $M$ is $\Sigma$-pure injective and
product-rigid (see \cite[Section 4.2]{Ang} and \cite[Section
3]{Kr-Sa}), and this implies $\Prod(M)=\Add(M)$. We remind that
that $\Prod(M)$ (resp. $\Add(M)$) denotes the class of all direct
summands in direct products (resp. direct sums) of copies of $M$.
We will characterize the converse inclusion under the usual set
theoretic assumption that there are no measurable cardinals. 
Here we use the terminology from \cite{Du-HZ}, hence a cardinal $\lambda=|I|$ is measurable 
if it is uncountable and there exists a 
countably-additive, non-trivial, $\{0,1\}$-valued measure $\mu$ on the power set of $I$ such that $\mu(I)=1$ 
and $\mu(\{x\})=0$ for all $x\in I$.

\begin{theorem} \label{ADD}  $(V=L)$
Let $R$ be a unital ring and $M$ be a right $R$-module. The
following are equivalent:
\begin{enumerate}[{\rm (1)}] \item $\Add(M)\subseteq \Prod(M)$ \item $M$ is
$\Sigma$-pure injective.\end{enumerate}
\end{theorem}

In particular, it is well known that a ring $R$ is right
pure-semisimple if and only if there exists a right $R$-module $M$
such that $\Modr R=\Add(M)$ (see \cite[Proposition 2.6]{Sto}). It
is also important to know when we can write $\Modr R=\Prod(M)$ for
some right $R$-module $M$. One of the reasons comes from the
theory of (co)homological functors on triangulated categories.
After A. Neeman introduced Brown representability for the dual,
\cite{Nee}, one of the main questions is to find examples of
triangulated categories $\mathcal{T}$ such that only one of the
categories $\mathcal{T}$ and $\mathcal{T}^{\mathrm{op}}$ satisfies
the Brown Representability Theorem (BRT). For instance, the
derived categories associated to some Grothendieck categories and
their duals satisfy BRT, cf. \cite{Alonso} and \cite{modoi2}. One
idea was to find a (non pure-semisimple) ring $R$ such that the
homotopy category of complexes $\Htp{\Modr R}$ satisfies BRT for
the dual, \cite{Mo-St}, i.e. to find a non pure-semisimple ring
such that $\Modr R=\Prod(X)$, cf. \cite{Mo}. By Theorem \ref{ADD}, this is not possible
under the set theoretic hypothesis $(V=L)$.

\begin{corollary} \label{BRT} $(V=L)$
Let $R$ be a unital ring. The following are equivalent:
\begin{enumerate}[{\rm (1)}]
\item $R$ is right pure-semisimple (i.e. all right $R$-modules are
pure-injective);


\item $\Htp{\Modr R}$ satisfies Brown Representability Theorem;

\item $\Htp{\Modr R}$ satisfies Brown Representability Theorem for
the dual;

\item there exists $X\in\Modr R$ such that $\Modr R=\Add(X)$;

\item there exists $X\in\Modr R$ such that $\Modr R=\Prod(X)$.
\end{enumerate}
\end{corollary}

\section{The proofs}

Let $R$ be a unital ring. The main tool used in the proof of our
theorem was proved in \cite{Du-HZ}. We remind that a functor
$P:\Modr R\to \Modr \Z$ is called a \textsl{$p$-functor} if it is
a subfunctor of the identity functor which commutes with direct
products. We refer to \cite{Hu} for basic properties of these
functors. We note here that every $p$-functor commutes with direct
sums, since it is a subfunctor for the identity functor. From the
same reason, it follows that if $P$ is a $p$-functor and $f:M\to
N$ is a homomorphism, then $P(f):P(M)\to P(N)$ is in fact the
restriction of $f$ to $P(M)$. These functors represent very useful
tools in the study of homomorphisms between direct products and
direct sums, e.g. \cite[Lemma 4]{BHZ79} and \cite[Theorem
2]{Du-HZ}. We start with a result which extends \cite[Theorem
2]{Du-HZ}. 

\begin{theorem}\label{Dugas-H}
Let $(U_i)_{i\in I}$ and $(V_j)_{j\in J}$ be two families of
$R$-modules such that the set $I$ is non-measurable. If
$f:\prod_{i\in I} U_i\to \bigoplus_{j\in J} V_j$ is a homomorphism
and $(P_n)_{n\in\mathbb{N}}$ is a descending chain of $p$-functors
defined on $\Modr R$ then there exist two finite subsets
$I'\subseteq I$, $J'\subseteq J$, and a positive integer $n_0$
such that $$\textstyle{f\left(P_{n_0}\left(\prod_{i\in I\setminus
I'}U_i\right)\right)\subseteq \bigoplus_{j\in J'}V_j+\bigcap_{n\in
\N}P_n\left(\bigoplus_{j\in J} V_j\right).}$$
\end{theorem}

\begin{proof}
The proof presented in \cite[Theorem 2]{Du-HZ} is valid for our
case, once we prove that the theorem is valid for the case when
$I$ is countable.

If $I$ is countable, we can apply \cite[Lemma 1]{Du-HZ}. In order
to do this we will prove that the topology defined on a direct
product $U=\prod_{i\in \mathbb{N}} U_i$ by the set
$$\textstyle{\mathcal{B}=\left\{\prod_{i\geq n}U_i\mid n\in \mathbb{N}\right\}}$$ as
a basis of neighborhoods of 0 is complete and Hausdorff. Note that
all $p$-functors are compatible with this topology, cf. the second
prototype described before \cite[Lemma 1]{Du-HZ}. It is not hard
to see that $\bigcap_{n\geq 0}\prod_{i\geq n}U_i=0$, hence this
topological space is Hausdorff.

For every $i\in \N$ we denote by $p_i:U\to U_i$ the canonical
projection. Let $(x^k)_{k\in\mathbb{N}}$ be a Cauchy sequence in
$U$. Then for every integer $n>0$ there exists an integer $u_n>n$
such that for all $v\geq u_n$ we have $x^v-x^{u_n}\in \prod_{i\geq
n+1}U_i$, which means that $p_i(x^v)=p_i(x^{u_n})$ for all $i\leq
n$. We can assume w.l.o.g. that the sequence $(u_n)_{n\in\N}$ is a
strictly increasing sequence of positive integers. If we consider
the element $x\in U$ defined by $p_i(x)=p_i(x^{u_i})$, it follows
that for every $n>0$ and for every $m>n$ we have
$p_i(x^{u_m}-x)=0$ for all $i\in\{1,\dots,n\}$. Then the
subsequence $(x^{u_n})_{n\in \N}$ of the sequence $(x^k)_{k\in
\N}$ converges to $x$. It follows that $(x^k)_{k\in \N}$ is
convergent since it is a Cauchy sequence. Therefore $U$ is
complete.

For the general case we sketch the proof, which is exactly the
same as in \cite{Du-HZ}, for reader's convenience. We
consider the set $\mathcal{I}$ of those subsets $T\subseteq I$ for
which the restriction of $f$ to $\prod_{i\in T} U_i$ satisfies our
claim, and suppose that $I\notin\mathcal{I}$.

We can prove, as in \cite{Du-HZ}, that $\mathcal{I}$ has the
following properties:
\begin{enumerate}[(a)]
\item the set $\mathcal{I}$ contains all finite subsets,

\item the set $\mathcal{I}$ is closed under subsets and finite
unions,

\item if $(J_n)_{n\in\mathbb{N}}$ is a family of pairwise disjoint
subsets of $I$ there exists $n_0\in\mathbb{N}$ such that
$\bigcup_{n>n_0}J_n\in\mathcal{I}$ (hence $\mathcal{I}$ will be
closed under countable unions).
\end{enumerate}
Therefore, we can construct a countably additive
$\{0,1\}$-valuated function on the power set of $I$, hence the
cardinality of $I$ is measurable. This leads to a contradiction,
and it follows that $I\in\mathcal{I}$.
\end{proof}

We also need the following elementary lemma:

\begin{lemma}
Let $V=\bigoplus_{j\in J}V_j$ be a right $R$-module. For every
$j\in J$ we fix a family of submodules $Y_j^n\leq V_j$,
${n\in\N}$. Then
$$\textstyle{\bigcap_{n\in\N} \left(\bigoplus_{j\in J}Y_j^n\right)= \bigoplus_{j\in
J}\left(\bigcap_{n\in\N}Y_j^n\right).}$$
\end{lemma}

\begin{proof}
If $x=(x_j)\in \bigoplus_{j\in J}V_j$ is in $\bigoplus_{j\in
J}Y_j^n$ for all $n\in \N$ (with $x_j\in Y_j$ for all $j\in J$)
then for all $j\in J$ and for all $n\in \N$ we have $x_j\in Y_j^n$
since the components $x_j$ are unique with respect to the direct
decomposition $\bigoplus_{j\in J}V_j$. Then for all $j\in J$ we
have $x_j\in \bigcap_{n\in\N}Y_j^n$, hence $x\in \bigoplus_{j\in
J}\left(\bigcap_{n\in\N}Y_j^n\right)$. Therefore $\bigcap_{n\in\N}
\left(\bigoplus_{j\in J}Y_j^n\right)\subseteq \bigoplus_{j\in
J}\left(\bigcap_{n\in\N}Y_j^n\right).$ The converse inclusion can
be proved in the same way.
\end{proof}

Now we can prove the main result of this section:

\begin{proposition}\label{main-prop}
Let $(U_i)_{i\in I}$ and $(V_j)_{j\in J}$ be two families of
non-zero (right) $R$-modules such that $|I|$ is infinite and
non-measurable and $|J|>|U_i|$ for all $i\in I$. Suppose that
there exists an epimorphism $f:\prod_{i\in I} U_i\to
\bigoplus_{j\in J}V_j$.

If $(P_n)_{n\in\N}$ is a descending family of $p$-functors defined
on $\Modr R$ such that the homomorphisms $P_n(f)$ are epimorphisms
for all $n$, then there exists an infinite subset $L\subseteq J$ such
that
for every $j\in L$ the sequence $P_n(V_j)$ is stationary.
\end{proposition}

\begin{proof}
By Theorem \ref{Dugas-H}, there exist two finite subsets
$I'\subseteq I$, $J'\subseteq J$, and a positive integer $n_0$
such that $$\textstyle{f\left(P_{n_0}\left(\prod_{i\in I\setminus
I'}U_i\right)\right)\subseteq \bigoplus_{j\in
J'}V_j+\bigcap_{n>0}P_n\left(\bigoplus_{j\in J} V_j\right).}$$
Since $P_{n_0}$ is a subfunctor for the identity functor,
$P_{n_0}(f)$ is the restriction of $f$ to
$P_{n_0}\left(\prod_{i\in I} U_i\right)$, hence
$f\left(P_{n_0}\left(\prod_{i\in I} U_i\right)\right)\subseteq
P_{n_0}\left(\bigoplus_{j\in J} V_j\right)$. Therefore,
$$\textstyle{f\left(P_{n_0}\left(\prod_{i\in I\setminus I'}
U_i\right)\right)\subseteq \left[\left(\bigoplus_{j\in
J'}V_j\right)+\bigcap_{n>0}P_n\left(\bigoplus_{j\in J}
V_j\right)\right]\bigcap P_{n_0}\left(\bigoplus_{j\in J}
V_j\right).}$$ Recall that the subgroup lattice of every abelian
group is modular. Using this we obtain
\begin{align*}\textstyle{ f\left( P_{n_0}\left(\prod_{i\in I\setminus I'}
U_i\right)\right) }  \subseteq & \textstyle{
\left[\left(\bigoplus_{j\in J'}V_j\right)\bigcap
P_{n_0}\left(\bigoplus_{j\in J}
V_j\right)\right]+\left[\bigcap_{n>0}P_n\left(\bigoplus_{j\in J}
V_j\right)\right] } \\  = & \textstyle{ \
P_{n_0}\left(\bigoplus_{j\in J'}
V_j\right)+\left[\bigcap_{n>0}P_n\left(\bigoplus_{j\in J}
V_j\right)\right] } \\ = & \textstyle{ \
P_{n_0}\left(\bigoplus_{j\in J'}
V_j\right)+\left[\bigcap_{n>0}\left(\bigoplus_{j\in J}
P_n(V_j)\right)\right]  } \\ = & \textstyle{ \
P_{n_0}\left(\bigoplus_{j\in J'} V_j\right)+\left[\bigoplus_{j\in
J}\left(\bigcap_{n>0} P_n(V_j)\right)\right]}\\ = & \textstyle{ \
P_{n_0}\left(\bigoplus_{j\in J'} V_j\right)+\left[\bigoplus_{j\in
J\setminus J'}\left(\bigcap_{n>0} P_n(V_j)\right)\right]. }
\end{align*}

Since $P_{n_0}(f):P_{n_0}\left(\prod_{i\in I} U_i\right)\to
P_{n_0}\left(\bigoplus_{j\in J} V_j\right)$ is an epimorphism, it
follows that $f$ induces an epimorphism of abelian groups
$$\overline{f}:\frac{P_{n_0}\left(\prod_{i\in I} U_i\right)}{P_{n_0}\left(\prod_{i\in I\setminus I'}
U_i\right)}\twoheadrightarrow \frac{P_{n_0}\left(\bigoplus_{j\in
J} V_j\right)}{P_{n_0}\left(\bigoplus_{j\in J'}
V_j\right)+\left[\bigoplus_{j\in J\setminus J'}\left(\bigcap_{n>0}
P_n(V_j)\right)\right]}\ .$$ But \begin{align*}
\frac{P_{n_0}\left(\bigoplus_{j\in J}
V_j\right)}{P_{n_0}\left(\bigoplus_{j\in J'}
V_j\right)+\left[\bigoplus_{j\in J\setminus J'}\left(\bigcap_{n>0}
P_n(V_j)\right)\right]} \cong & \frac{\bigoplus_{j\in J\setminus
J'} P_{n_0}\left( V_j\right)}{\left.\bigoplus_{j\in J\setminus
J'}\left(\bigcap_{n>0} P_n(V_j)\right)\right.}\\ \cong &
\bigoplus_{j\in J\setminus J'} \frac{ P_{n_0}\left(
V_j\right)}{\bigcap_{n>0} P_n(V_j)}\ .
\end{align*}
It follows that \begin{align*}\textstyle{\left|\bigoplus_{j\in
J\setminus J'} \frac{P_{n_0}\left( V_j\right)}{\bigcap_{n>0}
P_n(V_j)}\right|\leq \left|P_{n_0}\left(\prod_{i\in I'}
U_i\right)\right|=\left| \frac{P_{n_0}\left(\prod_{i\in I}
U_i\right)}{P_{n_0}\left(\prod_{i\in I\setminus
I'}U_i\right)}\right|},\end{align*} hence the cardinality of the
set $$L=\textstyle{\left\{j\in J\setminus J' \mid P_{n_0}(V_j)=
\bigcap_{n>0} P_n(V_j)\right\}}$$
is infinite since $|J|>\left|\prod_{i\in I'} U_i\right|\geq
\left|P_{n_0}\left(\prod_{i\in I'} U_i\right)\right|$.
\end{proof}

Recall from \cite[Observation 3]{Hu} that the \textsl{finite
matrix functors} are precisely the functors $\Hom_R(Z,-)(z):\Modr
R\to \Modr \Z$ with finitely presented $Z$ and $z\in Z$.

\begin{corollary}\label{corV}
Let $(U_i)_{i\in I}$ and $(V_j)_{j\in J}$ be two families of
(right) $R$-modules such that $J$ is infinite, $|J|>|U_i|$ for all
$i\in I$, and all modules $V_j$ are isomorphic to a fixed module
$V$.

Suppose that $I$ is non-measurable and there exists a pure
epimorphism $$\textstyle{f:\prod_{i\in I} U_i\to \bigoplus_{j\in
J}V_j.}$$ Then $V$ is $\Sigma$-pure injective.
\end{corollary}

\begin{proof}
Since finitely presented modules are projective with respect to
pure exact sequences, it is easy to see that all finite matrix
functors map pure epimorphisms to epimorphisms. Therefore, every
decreasing sequence of matrix functors is stationary on $V$. By
\cite[Theorem 6]{Hu}, $V$ is $\Sigma$-pure injective.
\end{proof}

\begin{remark}\label{rem-split}
If the homomorphism $f$ in the above results, Proposition
\ref{main-prop} and Corollary \ref{corV}, is split-epi then the
proofs are valid for all descending families of $p$-functors.
\end{remark}

\noindent{\textsl{The proof for Theorem \ref{ADD}.} Since
$\Add(M)\subseteq \Prod(M)$, it follows that for all sets $J$ the
right $R$-module $M^{(J)}$ is a direct summand of a direct product
of copies of $M$. If we take $J$ of cardinality greater than the
cardinality of $M$, the conclusion follows from Corollary
\ref{corV} (or Remark \ref{rem-split}) since all cardinals are
non-measurable. \qed

\medskip

\noindent{\textsl{The proof for Corollary \ref{BRT}.}
(1)$\Leftrightarrow$(2) is proved in \cite[Proposition 2.6]{Sto},
(2)$\Leftrightarrow$(4) and (3)$\Rightarrow$(5) follow from
\cite[Theorem 1 and Theorem 2]{Mo-St}. (1)$\Rightarrow$(3) follows
from \cite[Theorem 5.2]{Sto}, see also \cite[Remark 11]{Mo}. The
implication (5)$\Rightarrow$(1) is a consequence of Theorem \ref{ADD}.
\qed

\medskip

In \cite[Question 12]{Ber}, G. Bergman asks if we can deduce that
the canonical embedding $R^{(\omega)}\to R^\omega$ splits provided
that $R^{(\omega)}$ is isomorphic to a direct summand in
$R^{\omega}$. Corollary \ref{corV} can be used to provide a
partial positive answer to this question. In order to do this, let us recall from \cite[Satz 3 and Satz 4]{Ulam}
that if $\kappa$ is a non-measurable cardinal then $2^\kappa$ is also non-measurable. 

\begin{corollary}
Let $M$ be a module of non-measurable cardinality $\kappa$. If
$M^{\left(2^\kappa\right)}$ can be embedded as a direct summand in
a direct product $M^\lambda$ of copies of $M$ such that $\lambda$
is non-measurable (in particular in $M^{2^\kappa}$) then $M$ is
$\Sigma$-pure injective, thus the canonical embedding
$M^{(\omega)}\to M^\omega$ splits.
\end{corollary}

Finally, we mention that the condition $\Modr R=\Prod(X)$ also
appears as a characterization of rings which are
copure-semisimple, i.e. all modules are copure injective. 
If $M$ is a right $R$-module, we denote by $E(M)$ its
injective envelope. We recall from
\cite{field}, \cite{Hiremath} and
\cite{Wis} that a module $M$ is
\textsl{finitely cogenerated} if its injective envelope $E(M)$ is
isomorphic to $E(S_1)\oplus\dots \oplus E(S_n)$, where $S_i$ are
simple modules. The module $M$ is \textsl{finitely copresented} if
$E(M)$ and $E(M)/M$ are finitely cogenerated, and a short exact
sequence is \textsl{copure} if every finitely copresented module
is injective relative to this exact sequence.  A module is
\textsl{copure injective} if it is injective relative to all copure
short exact sequences. We mention that direct products of copure
injective modules are copure injective, \cite[Proposition
3]{Hiremath}.

We obtain an interesting corollary, concerning these rings:

\begin{corollary} $(V=L)$
A ring $R$ is right pure-semisimple if and only if all right
$R$-modules are copure injective.
\end{corollary}

\begin{proof}
Suppose that $R$ is right pure-semisimple. Using \cite[Theorem
4.5.4]{Prest} we observe that all finitely presented modules are
finitely copresented, hence they are copure injective. But all
modules are direct sums of finitely presented modules. Moreover,
since all pure exact sequences split, it follows that every
canonical embedding $0\to \bigoplus_{i\in I}{U_i}\to \prod_{i\in
I} U_i$ splits, hence every module is a direct summand of a copure
injective module. Therefore every module is copure injective.

Conversely, if every module is copure injective, then $\Modr
R=\Prod M$, where $M$ is the direct product of a set of
representatives of the isomorphism classes of all cofinitely
related modules, \cite[Theorem 7]{Hiremath}. The conclusion
follows from Corollary \ref{BRT}.
\end{proof}


\subsection*{Acknowledgements:} I would like to thank to
George-Ciprian Modoi who suggested me this topic, with whom I had
useful discussions on Brown Representability Theorem. I also thank
to the referee for his/her very useful critical comments.

 \end{document}